\documentclass{amsart}

\usepackage{latexsym}
\usepackage{rotating}
\usepackage{amsfonts}
\usepackage{amssymb}
\usepackage{amsmath}
\usepackage{graphics, color}
\definecolor{darkblue}{rgb}{0,0,0.4}
\usepackage[colorlinks=true, citecolor=darkblue, filecolor=darkblue,
linkcolor=darkblue, urlcolor=darkblue]{hyperref}
\input xy
\xyoption{all}

\DeclareMathOperator*{\mapsmo}{\maps}

\newtheorem{theorem}{Theorem}[section] % 1st argument is your name for it
\newtheorem{lemma}[theorem]{Lemma}     % 2nd argument is what is printed
\newtheorem{corollary}[theorem]{Corollary}
\newtheorem{proposition}[theorem]{Proposition}
\newtheorem{remark}[theorem]{Remark}
\newtheorem{definition}[theorem]{Definition}

\newcommand{\rk}{\mathrm{Rank}}

\newcommand{\bbC}{\mathbb{C}}
\newcommand{\C}{\mathbb{C}}

\newcommand{\bbN}{\mathbb{N}}
\newcommand{\R}{\mathbb{R}}
\newcommand{\bbR}{\mathbb{R}}

\newcommand{\Z}{\mathbb{Z}}
\newcommand{\bbZ}{\mathbb{Z}}

\newcommand{\bbQ}{\mathbb{Q}}
\newcommand{\ignore}[1]{}

\newcommand{\G}{\mathcal{G}}

\newcommand{\mH}{\mathcal{H}}
\newcommand{\mcH}{\mathcal{H}} %Holonomy map

\newcommand{\mcP}{\mathcal{P}}

\newcommand{\T}{\mathcal{T}}

\newcommand{\leqs}{\leqslant}
\newcommand{\geqs}{\geqslant}
\newcommand{\heq}{\simeq}

\newcommand{\maps}{\longrightarrow}
\newcommand{\lmaps}{\longleftarrow}
\newcommand{\injects}{\hookrightarrow}
\newcommand{\homeo}{\cong}

\newcommand{\isom}{\cong}
\newcommand{\cross}{\times}

\newcommand{\wt}[1]{\widetilde{#1}} %wide tilde for M's
\newcommand{\fc}{\mathcal{A}_{\mathrm{flat}}} %Space of flat connections

\newcommand{\Ch}{\mathrm{Ch}}

\newcommand{\Hom}{\mathrm{Hom}}

\newcommand{\GL}{\mathrm{GL}}

\newcommand{\hofib}{\mathrm{hofib}}

\newcommand{\Map}{\mathrm{Map}}
\newcommand{\bMap}{\mathrm{Map_*}}

\newcommand{\flatc}{\mathcal{A}_{\mathrm{flat}}}

\newcommand{\Id}{\mathrm{Id}}

\newcommand{\xmaps}{\xrightarrow}
\newcommand{\srm}[1]{\stackrel{#1}{\maps}}
\newcommand{\srt}[1]{\stackrel{#1}{\to}}
\newcommand{\sm}{\wedge}

\newcommand{\goesto}{\mapsto}
\newcommand{\nd}{\noindent}

\newcommand{\cc}{\Box}

\def\co{\colon\thinspace}

\newcommand{\Span}{\mathrm{Span}}
\newcommand{\Img}{\mathrm{Im}}
 
\newcommand{\e}{\emph}

 \newcommand{\ra}{\rangle}
\newcommand{\la}{\langle}

\newcommand{\coker}{\mathrm{coker}}

\title[Maps into algebraic sets and spaces of flat connections]% end with percent
 {Smoothing maps into algebraic sets and spaces of flat connections} % This is the full title of the paper

\author{Thomas Baird and Daniel A. Ramras}

\thanks{The first author was partly supported by an NSERC discovery grant. The second author was partially supported by NSF grants DMS-0804553 and DMS-0968766 and a Collaboration Grant from the Simons Foundation.\\
2010 \e{Mathematics Subject Classification}.  Primary 14P05, 53C05; Secondary 57R20, 55R37}

\begin{document}

\begin{abstract}
Let $X\subset \bbR^n$ be a real algebraic set and $M$ a smooth, closed manifold.  We show that all continuous maps $M\to X$ are homotopic (in $X$) to $C^\infty$ maps.  
We apply this result to study characteristic classes of vector bundles associated to continuous families of complex group representations, and we establish lower bounds on the ranks of the homotopy groups of spaces of flat connections over aspherical manifolds.
\end{abstract}

%\part{Use this type of header for very long papers only}
% use lowercase except for proper names
\maketitle

\section {Introduction}

The first goal of this paper is to prove the following result about the differential topology of algebraic sets.

\begin{theorem}[Section~\ref{smoothing-sec}]\label{smoothing theorem}  Let $X\subset \bbR^n$ be a (possibly singular) real algebraic set, and let $f\co M\to X$ be a continuous map from a smooth, closed manifold $M$. Then there exists a map $g\co M\to X$, and a homotopy $H\co M\cross I \to X$ connecting $f$ and $g$, such that the composite $M\srt{g} X\injects \bbR^n$ is $C^\infty$.
\end{theorem}

The problem of smoothing maps into algebraic sets seems natural, but we have not found mention of it in the literature. We consulted several experts in real algebraic geometry; some expected our result to hold, and some did not.\footnote{The question of whether a continuous map from a compact smooth variety $X$ into a variety $Y$ is homotopic to a \e{regular} map has been considered by several authors, and in general the answer is negative~\cite{Bochnak-Kucharz, Ghiloni, Kucharz-Maciejewski}.}

Our proof proceeds by embedding $X$ as the singular set of an irreducible, quasi-projective variety $Y$ and using a resolution of singularities $\wt{Y}\to Y$ for which the inverse image of $X$ is a divisor with normal crossing singularities.  Basic facts about neighborhoods of algebraic sets then reduce the problem to the case of normal crossing divisors, which can be handled by differential-geometric means.  We note that a small modification of our arguments can prove a version of this result for smoothing out maps from non-compact manifolds into compact algebraic sets.

In the remainder of the paper, we apply Theorem \ref{smoothing theorem} to study the topology of  spaces of flat connections over aspherical manifolds and their associated moduli spaces. In particular, given a discrete group $\Gamma$ such that  $B\Gamma$ has the homotopy type of a finite CW complex, we study the functorial map 
\begin{equation}\label{intial}
\Hom(\Gamma, GL_n(\C)) \srm{B} \Map_* (B\Gamma, BGL_n(\C)),
\end{equation}
from the space of complex-valued group homomorphisms to the space of based continuous maps between classifying spaces. If $B\Gamma$ is homotopy equivalent to a smooth manifold, then (\ref{intial}) may be interpreted as the inclusion of the moduli space of based flat connections into the moduli space of based connections. We prove 

\begin{theorem}[Section~\ref{TAS}]$\label{cokernel-AS}$ Let $\Gamma$ be a discrete group such that $B\Gamma$ has the homotopy type of a finite, $d$--dimensional CW complex and let $\beta_k (B\Gamma) = \rk (H^k (B\Gamma; \bbQ))$.  Then for each $\rho\in \Hom(\Gamma, \GL_n (\bbC))$ and each $m\geqs 0$, the induced map on homotopy groups
$$B_*: \pi_m \left(\Hom(\Gamma, \GL_n (\bbC)), \rho\right) \to \pi_m \left(\bMap (B\Gamma, B\GL_n (\bbC)), B\rho\right)$$
satisfies $\rk (\coker (B_*)) \geqs \sum_{i = 1}^\infty \beta_{m+2i} (B\Gamma)$
so long as  $n\geqs (m+d)/2$. 
\end{theorem}

The strategy behind the proof of Theorem \ref{cokernel-AS} is as follows. The space $BGL_n(\C)$ classifies $\C^n$-vector bundles, so an element in $\pi_m \left(\bMap (B\Gamma, B\GL_n (\bbC)), B\rho\right)$ determines a vector bundle (up to isomorphism) over $S^m \cross B\Gamma$.  Those classes lying in the image of $B_*$ correspond to vector bundles over $S^m \cross B\Gamma$ which admit a connection that is ``flat in the $B\Gamma$ direction."  We then deduce, using Chern-Weil theory, that certain Chern classes of these bundles must vanish (see Theorem \ref{vanishing2}). Theorem \ref{smoothing theorem} is necessary to make this argument rigorous:
to construct the desired connection on the bundle associated to $B_* ([\psi])$, we must smooth out the underlying map $\psi$ from $S^m$ into the (usually singular) algebraic set $\Hom(\Gamma, \GL_n (\bbC))$.

In Section \ref{flat}, we prove the following consequence of Theorem \ref{cokernel-AS} (recall that $U(n)$ is the maximal compact subgroup of $GL_n(\C)$).

\begin{corollary}[Section~\ref{flat}]$\label{fh}$
Let $M$ be a closed, smooth, aspherical manifold of dimension $d$, and let $E\to M$ be a flat principal $U(n)$--bundle over $M$.  Let $\fc(E)$ denote the space of flat, unitary connections on $E$. Then for each $A_0 \in \fc (E)$ and for each $0<m \leqs 2n-d-1$, 
we have 
$$\rk\left(\pi_m (\fc(E), A_0)\right) \geqs \sum_{i=1}^\infty \beta_{m+2i+1} (M).$$
Additionally, if $\sum_{i=1}^\infty \beta_{2i+1} (M) > 0$ and $2n-d-1\geqs 0$, then $\fc (E)$ has infinitely many path components.
\end{corollary}

In particular,  the space of flat connections on a flat $U(n)$--bundle ($n\geqs 2$) over an orientable aspherical 3--manifold has infinitely many path components.  This result appears to be new, even for the trivial bundle on the 3--torus, and is in contrast with  the case of bundles over surfaces.  There, Yang--Mills theory shows that $\flatc(E)$ is highly connected with respect to the rank of $E$ (see Ramras~\cite[Proposition 4.9]{Ramras-surface}, for instance).  One fundamental difficulty in proving Corollary~\ref{fh} is that the space of flat connections is not smooth (even in an infinite-dimensional sense) and hence it is unclear whether paths in this space are always homotopic to smooth paths.

The methods developed in this paper (Theorem~\ref{cokernel-AS} in particular) also have applications to the study of Quillen--Lichtenbaum phenomena in the stable topology of representation spaces, as observed for products of surface groups in Ramras~\cite{Ramras-stable-moduli}.  In fact, the desire to understand these phenomena gave rise to the present work.  These topics will be discussed in detail elsewhere.

\vspace{.1in} 

\noindent {\bf Organization:} In Section~\ref{smoothing-sec}, we establish our results about maps into real algebraic sets, which are applied to study the characteristic classes in Section~\ref{char-classes}.  Theorem~\ref{cokernel-AS} is proven in Section~\ref{TAS}, and spaces of flat connections are studied in Section~\ref{flat-sec}.
 
\vspace{.1in} 

\noindent {\bf Acknowledgements}\label{ackref}
The first author thanks Juan Souto and Frances Kirwan for helpful conversations.  The second author thanks Ben Wieland.  In addition, the suggestions of the anonymous referees helped to improve the exposition.

%%%%%%%%%%%%%%%%%%%%%%%%%%%%%%%%%%%%%%%%%%%%%%%%%%%
%%%%%%%%%%%%%%%%%%%%%%%%%%%%%%%%%%%%%%%%%%%%%%%%%%%

\section{Smoothing maps into real algebraic sets}$\label{smoothing-sec}$
The goal of this section is to prove the following result.

\begin{theorem}\label{smoothing}
Let $i: X \hookrightarrow \R^n$ be a real algebraic set and let $M$ be a smooth manifold.   Assume that either $X$ or $M$ is compact.  Then for every continuous map $\phi:M \rightarrow X$ there exists a map $\phi': M \rightarrow X$ homotopic to $\phi$ such that $i \circ \phi'$ is smooth.\footnote{We use Euclidean topology throughout, except that ``irreducible algebraic set" means irreducible in the Zariski topology.  We will identify all real algebraic sets with their real points throughout this section.}
\end{theorem}

The following example helps illustrate the subtlety of Theorem \ref{smoothing}.  Consider the subspace $X \subset \R^2$  defined as the union of the graph 
$$\{(x,y)\, |\, y= x \sin(1/x), x\in [-1/\pi, 0) \cup (0, 1/\pi]\}\cup \{(0,0)\}$$ 
with the polygonal path traveling from $(-1/\pi, 0)$ to $(-1/\pi, 1)$ to $(1/\pi, 1)$ to $(1/\pi, 0)$.  The subspace $X$ is homeomorphic to $S^1$, so $\pi_1(X) \cong \Z$. However, the non-trivial elements of $\pi_1(X)$ cannot be represented by  smooth maps $\phi:S^1 \rightarrow X$ because the curve $x\sin (1/x)$ has infinite arc length.  Note, moreover, that  just as with the real algebraic sets considered in Theorem~\ref{smoothing}, it is possible to triangulate $\bbR^2$ with $X$ as a subcomplex.  This follows from the Schoenflies Theorem (see, for instance, Moise~\cite[Chapter 10, Theorem 4]{Moise}) but can in fact be done quite explicitly.

For the proof of Theorem \ref{smoothing}, we  assume that $M$ is compact, which is the only case needed  in the paper; the argument is similar when $X$ is compact but $M$ is not.

\begin{lemma}
Let $X \subset \R^n$ be a real algebraic set. There exists an irreducible real algebraic set $Y \subset \R^{n+2}$ such that $Y$ is homeomorphic to $\R^{n+1}$ and the projection map  $\pi \co \R^{n+2} = \R^n \oplus \R^2\to \R^n$ restricts to a homeomorphism $$\pi: Sing(Y) \srm{\cong} X.$$
\end{lemma}

\begin{proof}
We may suppose that $X$ is equal to the zero set of a single polynomial $p = p(x_1,\ldots ,x_n)$, that is: 
$$X = Z(p).$$

Introduce new variables $y_1, y_2$ and consider the polynomial $g(x,y) = p(x)^2 + y_1^2 + y_2^3$. Define $$ Y = Z(g).$$  It is not hard to verify that the corresponding inclusion $\R^n \hookrightarrow \R^{n+2}$ identifies $X$ bijectively with the singular locus of $Y$, so the projection onto $\R^n$ identifies $ Sing(Y) \cong X$.

Given $x_1,\ldots ,x_n, y_1\in \bbR$, the formula $y_2 = ( g(x,y) - p(x)^2 - y_1^2)^{1/3}$ gives the 
unique solution to $g(x,y) = 0$.  Hence $Y$ is homeomorphic to $\R^{n+1}$. Furthermore, the polynomial $g(x,y)$ is irreducible, since its reduction $y_1^2 + y_2^3$ is irreducible (see, for example, Dummit and Foote~\cite[Section 9.1]{Dummit-Foote}).
\end{proof}

In particular, a map into $X$ is smooth as a map into $\R^n$ if and only if the corresponding map into $Sing(Y)$ is smooth as a map into $\R^{n+2}$. We identify $X$ with $Sing(Y)$ in what follows.  The following version of Hironaka's resolution of singularities theorem appears in Koll{\'a}r~\cite[Theorem 3.27]{Kollar}.

\begin{theorem}\label{arlcboeluracb}
Given a quasiprojective real algebraic variety $Y \subset \R^{n+2}$ with singular set $X$, there exists a nonsingular real algebraic variety $\wt{Y}$ and a birational and projective morphism $$ \Pi: \wt{Y} \rightarrow Y$$ which restricts to an isomorphism over $Y - X$ and for which $\wt{X} := \Pi^{-1}(X)$ is a divisor with simple normal crossings.
\end{theorem}

Morphism here means regular map, which in particular means that $\Pi$ is smooth as a map $\wt{Y}\to \R^{n+2}$. The statement that $\Pi$ is projective implies that it is proper in the Euclidean topology (this is well known over $\C$ and follows over $\R$ by taking fixed points of complex conjugation). The simple normal crossing condition (plus properness) implies that $\wt{X}$ is a finite union of embedded, boundary-free, connected submanifolds of codimension 1, intersecting transversely.

\begin{lemma}
Let notation be as in Theorem \ref{arlcboeluracb}, and let $M$ be a finite simplicial complex. Then for every continuous map $f: M \rightarrow X$, there exists a continuous map $g: M \rightarrow \wt{X}$ for which $\Pi \circ g$ is homotopic to $f$.
\end{lemma}

\begin{proof}  Choose $Y$ as in the proof of Theorem~\ref{arlcboeluracb}.
Since $M$ is compact, the image $f(M) \subseteq X \subseteq Y \subseteq \R^{n+2}$ is also compact.  Choose a sufficiently large closed ball $B \subset \R^{n+2}$ so that $f(M) \subseteq B\cap X$.  Then $B\cap X$ is a compact semi-algebraic subset of  $\R^{n+2}$, so by Durfee \cite{Durfee} there exists an open $U\supset B\cap X$ for which inclusion is a homotopy equivalence. Furthermore,  $\Pi^{-1}( B \cap X)$ is a compact, semi-algebraic subset of the non-singular variety $\wt{Y}$ so using the method of rug functions from \cite{Durfee}, $U$ may be chosen so that $\Pi^{-1}(U)$ deformation retracts to $\Pi^{-1}( B \cap X)$.  We are left with a commuting diagram

\begin{equation}\label{rcebarcuebra} \xymatrix{  \Pi^{-1}(B\cap X) \ar[d] \ar[r] & \ar[d] \Pi^{-1}(U)\\
                 B\cap X \ar[r] & U    } \end{equation}
where the horizontal arrows are homotopy equivalences and the vertical maps are restrictions of $\Pi$.\footnote{One may establish the existence of such a diagram using the theory of regular neighborhoods, as presented in Hudson--Zeeman~\cite{Hudson-Zeeman}, instead of the theory of rug functions.}  Since $f(M) \subset B\cap X$ is compact, for a sufficiently small $\epsilon >0$,  $f$ is homotopic, as a map into $U$, to
$$f_{\epsilon}: M \rightarrow U, ~~f_{\epsilon} (m) = (f(m), \epsilon^3, -\epsilon^2).$$ 
In fact, since $f_{\epsilon}(M) \subset Y \setminus X$ and $\Pi$ restricts to homeomorphism $\wt{Y} \setminus \wt{X} \cong Y\setminus X$, there exists a unique map $g_{\epsilon}: M \rightarrow \Pi^{-1}(U)$ with $\Pi \circ g_{\epsilon} = f_{\epsilon}$.  Because the horizontal arrows of (\ref{rcebarcuebra}) are homotopy equivalences, this implies the existence of $g$.
\end{proof}

To summarize: For each real, affine set $X$ and each map $f:M \rightarrow X$ we may lift to a map $g : M \rightarrow \wt{X}$ such that $\Pi \circ g$ is homotopic to $f$ (in $X$). If $g$ can be deformed inside $\wt{X}$ to a smooth map $h$, then $\Pi \circ h$ is smooth and homotopic to $f$. Thus to prove Theorem \ref{smoothing} it suffices to prove the special case that $X$ is a simple normal crossing divisor. In terms of the underlying smooth manifolds,  we may restrict to the case that $Y$ is a smooth manifold without boundary and $ X \subset Y$ is a union $$X = \bigcup_{i=1}^N M_i$$ where each $M_i$ is an embedded, connected manifold (without boundary) of codimension one in $Y$, and $Y$ may be covered by coordinate neighbourhoods whose intersection with $X$ is a union of coordinate hyperplanes. 

\begin{lemma}\label{lcabolecruba}
Let $Y$ be a non-singular irreducible real algebraic set and let $X \subset Y$ be a simple normal crossing divisor as above. For each $i \in 1,\ldots ,N$, let $$\pi_i: \nu_i \rightarrow M_i$$ be the (one-dimensional) normal bundle of $M_i$ in $Y$. Then there exists a smooth tubular neighbourhood map $$\phi_i: \nu_i \hookrightarrow Y$$
such that for  all $ j \not=i$, we have $\phi_i(n) \in M_j \Leftrightarrow \phi_i(\pi_i(n)) \in M_j$. 

\end{lemma}
\begin{proof}
Consider first the situation where the normal bundle $\nu_i$ is trivial, so that $\nu_i \cong M_i \times \R$. Choose an integrable, smooth vector field $V \in \mathfrak{X}(Y)$, which is transverse to $M_i$ but is tangent to $M_j$ for all $j \not=i$. Because $\cup_{i=1}^n M_i$ looks locally like a union of coordinate hyperplanes, such a $V$ is easily constructed by gluing together locally defined vector fields using a partition of unity. Then $V$ integrates a one-parameter group of diffeomorphisms $\Phi: Y \times \R \rightarrow Y$ preserving $M_j$ for $j \not=i$. The restriction $ \Phi|_{M_i\times \R}: M_i \times \R \rightarrow Y$ sends $M_i \times \{0\}$ diffeomorphically onto $M_i$  and is regular along $M_i \times \{0\}$ because $V$ is transverse to $M_i$. By a version of the inverse function theorem (Guillemin-Pollack \cite[1.8 Exercise 14]{GP}), $\Phi|_{M_i\times \R}$ restricts to a diffeomorphism between a tubular neighbourhood of $M_i\times \{0\}$ in $M_i\times \R$ and a tubular neighbourhood of $M_i$ in $Y$. Identifying $\nu_i$ with a tubular neighbourhood in $M_i\times \R$ completes the argument.

Now consider the case that $\nu_i$ is non-orientable. Let $\rho: \wt{M}_i \rightarrow M_i$ be a double covering map with $\rho^*\nu_i$ trivial ( $\wt{M}_i$ may be taken to be the unit sphere bundle in $\nu_i$ for some metric on $\nu_i$). Let $ M_i \subset U$ be an open tubular neighbourhood of $M_i$ in $Y$. Taking the corresponding cover of $U$, we may form the commutative diagram

$$
	\xymatrix{  \wt{M_i} \ar[r] \ar[d]_{\rho} & \wt{U} \ar[d]^{\rho_U} \\
	              M_i  \ar[r] & U } $$
where $\rho$ and $\rho_U$ are two-fold coverings with deck transformation $\tau$, and the horizontal maps are $\tau$-equivariant embeddings for which $\wt{M_i}$ has trivial normal bundle. The preimage $\rho_U^{-1}(X \cap U)$ is a normal crossing divisor in $\wt{U}$ containing $\wt{M_i}$. 

As before, choose a smooth vector field $V \in \mathfrak{X}(U)$ which is transverse to $\wt{M_i}$ and tangent to $\rho_U^{-1}(M_j)$ for $j \not= i$. Replacing $V$ by $\frac{1}{2}(V - \tau_*(V))$, we may assume that $\tau_*(V) = -V$. 				Proceeding as before, we obtain a $\tau$-equivariant submersion $\wt{\phi}_i :\rho^*\nu_i \rightarrow \wt{U}  $. This descends to a tubular neighbourhood map
$$\phi_i: \nu_i \cong (\rho^*\nu_i) / \tau \rightarrow U \subseteq Y$$
satisfying the requirements of the theorem.

\end{proof}

\begin{proof}[Proof of Theorem \ref{smoothing}]
	It suffices to consider the case of that $X = \bigcup_{i=1}^N M_i$ is a simple normal crossing divisor in a non-singular, irreducible real algebraic set $Y$.  The proof proceeds by constructing an open (weak) deformation retraction neighbourhood $U$ of $X$ such that the retraction $r: U \rightarrow X$ is smooth as a map to $Y$. Since any continuous map $f:M \rightarrow X$ is homotopic to a smooth map $f': M \rightarrow U$, the composition $r \circ f'$ is smooth and homotopic to $f$.
	
Choose a tubular neighbourhood $ \phi_i: \nu_i \rightarrow Y$ for each $M_i$ as in Lemma \ref{lcabolecruba} and choose a smooth inner product on 
$\nu_i$. Define the map $h_i:\nu_i \rightarrow \nu_i$  by $ (m, v) \rightarrow (m, s(|v|^2) v)$ where $s: \R \rightarrow \R$ is a smooth map such that $s(x) =0 $ for $|x| <1$ and $s(x) = 1$ for $|x| >2$. Extending $\phi_i \circ h_i \circ \phi_i^{-1}$ by the identity map determines a smooth map $H_i: Y \rightarrow Y$ such that 
\begin{itemize}
\item $H_i(U_i) = M_i$ for some open neighbourhood $U_i$ of $M_i$,
\item $H_i(X) \subseteq X$, and
\item $H_i|_X$ is homotopic to the identity on $X$. 	
\end{itemize}
The composition $H_1 \circ \cdots\circ H_N$ sends the open neighbourhood $$V = \bigcup_{i=1}^N (H_{i+1} \circ \cdots\circ H_N)^{-1}(U_i)$$ to $X$. Taking any (possibly smaller) neighbourhood $X \subset U \subseteq V$ such that $U$ deformation retracts to $X$ completes the proof.
\end{proof}

%%%%%%%%%%%%%%%%%%%%%%%%%%%%%%%%%%%%%%%%
%%%%%%%%%%%%%%%%%%%%%%%%%%%%%%%%%%%%%%%%

\section{Characteristic classes of flat families}$\label{char-classes}$
Let $\Gamma$ be a discrete group, let $G$ be a Lie group, and let $\Hom(\Gamma, G)$ denote the space of all  homomorphisms from $\Gamma$ to $G$, with the subspace topology inherited from product topology on $G^\Gamma$.   A continuous map $\rho\co X\to \Hom(\Gamma, G)$ will be called an $X$--\e{family} of representations (or simply an $X$--family), and we set $\rho_x = \rho(x)$.  
We say that an $X$--family $\rho$ is smooth if $X$ is a smooth manifold and for each $\gamma\in \Gamma$, the map $X\to G$ given by $x\goesto \rho_x (\gamma)$ is smooth.

We now associate a principal $G$--bundle  to each $X$--family.
Given a based CW complex $(N, n_0)$, let $\wt{N}$ denote the universal cover of $N$.  We view $\wt{N}$ as the set (appropriately topologized) of based homotopy classes $\la l \ra$ of paths $l\co [0, 1]\to N$ with $l(1) = n_0$ (the opposite  convention is commonly used, e.g. in Hatcher~\cite[Section 1.3]{Hatcher}).  The constant path at $n_0$ will be denoted $\wt{n}_0$, and the projection 
$q\co (\wt{N}, \wt{n}_0)\to (N, n_0)$
 given by evaluation at $0$ is a right principal $(\pi_1 (N, n_0))$--bundle, with right action given by $\la l\ra \cdot \la\alpha\ra=\la l\cc \alpha\ra$ for any loop $\alpha$ based at $n_0$ (here $l\cc \alpha$ denotes the concatenated loop tracing out $l$ first and then $\alpha$).

Associated to an $X$--family $\rho$ of representations of $\pi_1 (N, n_0)$, we can now build a right principal $G$--bundle 
$$E_\rho (N) = (X\cross \wt{N} \cross G)/ \pi_1 (N, n_0)\srm{\pi} X\cross N$$
where $\gamma\in \pi_1 (N, n_0)$ acts by $(x, \wt{n}, g)  \cdot  \gamma = (x, \wt{n}  \cdot \gamma , \rho_x (\gamma)^{-1} g)$, and $\pi$ is given by $\pi([x, \wt{n}, g]) = (x, q(\wt{n}))$.  If $X$ and $N$ are smooth manifolds and $\rho$ is a smooth $X$--family, then this action of $\pi_1 (N, n_0)$ is smooth, properly discontinuous, and $G$--equivariant, so the orbit space $E_\rho (N)$ inherits (in a canonical fashion) the structure of a smooth principal  $G$--bundle over $(X \times \wt{N})/\pi_1 (N, n_0) \cong X \times N$.

If $\Gamma$ is a discrete group, we set $E_\rho := E_\rho (B\Gamma)$, where $B\Gamma$ is the simplicial model for the classifying space of $\Gamma$ (i.e. $B\Gamma$ is the geometric realization of $\Gamma$ viewed as a category with one object).
Let $E\Gamma \to B\Gamma$ denote the simplicial model for the universal principal $\Gamma$--bundle.  We use the conventions described in Ramras~\cite[Section 2]{Ramras-excision}, so that this is a right principal $\Gamma$--bundle.  Universality of $E\Gamma$ yields a $\Gamma$--equivariant homeomorphism $\wt{B\Gamma} \to E\Gamma$ (covering the identity on $B\Gamma$).  
Hence we have an isomorphism of principal $G$--bundles over $B\Gamma$:
$$E_\rho = (X\cross \wt{B\Gamma} \cross G)/\Gamma \isom (X\cross E\Gamma \cross G)/\Gamma.$$
This alternate description of $E_\rho$ will be useful due to its relationship with the (simplicial)  universal principal $G$--bundle $EG \to BG$.

We now record some basic properties of these constructions.

\begin{lemma}\label{compatible} 
Let $(N, n_0)$ and $(N', n'_0)$ be based, connected CW complexes and let $\rho$ and $\rho'$ be $X$--families of representations of $\pi_1 (N, n_0)$.

\begin{enumerate}
\item Given maps $(N', n'_0) \srt{g} (N, n_0)$ and $X' \srt{f} X$,
there is an isomorphism 
$$(f\cross g)^* (E_\rho (N)) \isom E_{g^\# \circ \rho\circ f} (N')$$
of principal $G$--bundles over $X'\cross N'$, where
$$g^\#\co \Hom(\pi_1 (N, n_0), G)\maps  \Hom(\pi_1 (N', n'_0), G)$$
is the map induced by composition with $g_* \co \pi_1 (N', n'_0) \to \pi_1 (N, n_0)$.

\item If $\rho$ and $\rho'$ are homotopic, then 
$E_\rho (N) \isom E_{\rho'} (N)$.

\end{enumerate}
\end{lemma}
\begin{proof} The map $f$ induces a map 
$\wt{f} \co (\wt{N}, \wt{n}_0)\maps (\wt{N'}, \wt{n'_0}),$
defined by $\wt{f} (\la l \ra) = \la f\circ l \ra$,
and $\wt{f}$ is equivariant in the sense that 
$\wt{f} (\wt{n}\cdot \gamma) = \left( \wt{f} (\wt{n})\right) \cdot f_* (\gamma)$
for each $\wt{n}\in \wt{N}$ and each $\gamma\in \pi_1 (N, n_0)$.  This proves (1).
Part (2) follows from the Bundle Homotopy Theorem, because if 
$H\co X\cross I\to \Hom(\Gamma, G)$
 is a homotopy from $\rho$ to $\rho'$, then 
 $E_{H} (N) \to X\cross I\cross N$
 is a bundle homotopy between $E_\rho (N)$ and $E_{\rho'} (N)$.
\end{proof}

Let $S$ be a finite generating set for $\Gamma$, with cardinality $s$, and let $G = \GL_n (\bbC)$.  Then the image of the restriction map $\Hom(\Gamma, \GL_n (\bbC))\to (\GL_n (\bbC))^s$ is a real algebraic subset of $(\GL_n (\bbC))^s$, which is itself a real algebraic subset of $\bbR^{4n^2 s}$ via the homeomorphism 
\begin{equation} \label{GLn}\GL_n (\bbC) \isom \{(A, B)\in (M_n (\bbC))^2 \,|\, AB = I_n\} \subset \bbC^{2n^2} \subset \bbR^{4n^2}.\end{equation}
If $X$ is a smooth manifold, Theorem~\ref{smoothing} guarantees that every $X$--family $\rho$  is homotopic (through $X$--families of representations) to a smooth $X$--family (the notion of smoothness from Section~\ref{smoothing-sec}, which depends on the embedding (\ref{GLn}) and the choice of generating set, agrees with the more intrinsic notion of smoothness introduced at the start of this section; this follows easily from the fact that inversion in $\GL_n (\bbC)$ is a smooth map).  The same is true for $G = U(n)$.

\begin{definition}$\label{fwf-def}$ If $X$ and $M$ are smooth manifolds and $E\to X\cross M$ is a smooth principal $\GL_n (\bbC)$--bundle, we say that a connection on $E$ is \e{fiberwise flat} if its restriction  to each slice $\{x\} \cross M$ is a flat connection $(x\in X)$.   We define fiberwise flat connections for families of $U(n)$--bundles in the analogous manner.  
(We think of  $E$ as a family, parametrized by $X$, of bundles over $M$.)
\end{definition}

\begin{lemma}$\label{fwf-conn}$
 If $X$ and $M$ are smooth manifolds  and 
 $$\rho\co X\to \Hom(\pi_1 (M, m_0), \GL_n (\bbC))$$
  is a smooth family of representations,
 then the bundle $E_{\rho} (M)$ admits a fiberwise flat connection whose holonomy over $\{x\} \cross M$ (computed at $[x, \wt{m_0}, I] \in E_{\rho} (M)$) is precisely $\rho(x)$.  The analogous result holds in the unitary case.
 \end{lemma}
 
 \begin{proof}
  In the Ehresmann formulation, a connection on a $\GL_n (\bbC)$-bundle $P$ is a $\GL_n (\bbC)$--invariant splitting $TP = V + H$ into a sum of smooth subbundles, where $V$ is the vertical bundle tangent to the fibers of $P$.  We call $H$  the horizontal bundle.

Let $\wt{E} = X \times \wt{M} \times \GL_n (\bbC)$, and let $\wt{D} \subset T\wt{E}$ be the smooth subbundle of tangents to the submanifolds $\{x\} \times \wt{M} \times \{A\}$ for all fixed $ (x,A) \in X \times \GL_n (\bbC)$. Then $\wt{D}$ is invariant under both $\GL_n (\bbC)$ and $\Gamma$, so it descends to a smooth, $\GL_n (\bbC)$-equivariant subbundle $ D \subset TE_{\rho} (M)$, which intersects the vertical bundle $V \subset TE_{\rho} (M)$ trivially.  The total space of   $E_{\rho} (M)$ admits a $\GL_n (\bbC)$--invariant Riemannian metric\footnote{To construct such a metric on a principal $G$--bundle $P$ over a manifold $N$, where $G$ is a Lie group, note that if $P$ is trivial over $U\subset N$ then 
$T(P|_{U}) \isom T(U) \cross T_e (G) \cross G,$ and
any choice of metric on $T(U) \oplus ( T_e (G)\cross U)$ gives rise, under   translation by $G$, to a $G$--invariant inner product on $T(P|_{U})$.  These inner products can be pasted together using a partition of unity on $N$.},
and relative to this metric, we define an Ehresmann connection 
$$H := D + (D + V)^{\perp}$$ 
on $E_{\rho} (M)$. The restriction of $H$ to a slice $\{x\} \times M \subset X \times M$ is a flat connection on $E_{\rho(x)}$, and we must show that the holonomy of this connection (computed at $[x, \wt{m_0}, I] \in E_{\rho(x)} (M)$) is  precisely
$\rho(x)$.
To check this, note that for any loop $\alpha\co ([0,1], \{0,1\}) \to (M, m_0)$, 
the horizontal lift of $\alpha$ to $E_{\rho} (M)$ is simply
$$t\goesto [x, \wt{\alpha} (t), I],$$
where $\wt{\alpha}$ is the unique lift of $\alpha$ to $\wt{M}$ satisfying $\wt{\alpha} (0) = \wt{m}_0$.  We now have
$$[x, \wt{\alpha}(1), I] \cdot \rho_x ([\alpha])= [x, \wt{\alpha}(1), \rho_x ([\alpha])] = [x, \wt{\alpha}(1) \cdot [\alpha]^{-1}, I] = [x, \wt{m}_0, I]$$
as desired.  (The last equality relies on the choice made in defining $\wt{M}$.)
\end{proof}

We   need the following basic result regarding cohomology in characteristic zero.

\begin{lemma}\label{Thom} Let $Y$ be a topological space.  Then there exists a set of smooth, closed manifolds $\{M_k\}_{k\in K}$, and a map $f\co M := \coprod_k M_k \to Y$ such that for each $j\geqs 0$,
$$f^*\co H^j (Y; \bbQ) \to H^j (M; \bbQ)\isom \prod_k H^j(M_k; \bbQ)$$
is injective.  If $H^j (Y; \bbQ) = 0$ for $j>m$, we can assume $\dim (M_k) \leqs m$ for all $k$.
\end{lemma}
\begin{proof} Choose a basis $\{x_k\}_k$ for $H_* (Y; \bbQ) = \bigoplus_{j\geqs 0} H_j (Y; \bbQ)$ as an $\bbQ$--vector space.  By Thom's theorem~\cite{Thom} (see also Ga{\u\i}fullin~\cite{Gaifullin-realization-2, Gaifullin-realization-1}) for each $k$ there exists a map $f_k\co M_k\to Y$, with $M_k$  a closed, smooth, orientable manifold, such that $(f_k)_*([M_k]) = q_k x_k$ for some non-zero $q_k\in \bbQ$ (where $[M_k]$ is the fundamental class), and if  $H^j (Y; \bbQ) = 0$ for $j>m$, then $\dim (M_k) \leqs m$.  Letting $(f_k)_*^\vee$ denote the $\bbQ$--linear dual of $(f_k)_*$, the Universal Coefficient Theorem for cohomology yields a  commutative diagram 
$$\xymatrix{ H^j(Y; \bbQ) \ar[r]^-\isom  \ar[d]^{\prod_k (f_k)^*}  & \Hom_\bbQ (H_j (Y; \bbQ), \bbQ) \ar[d]^{\prod_k (f_k)_*^\vee} \\
		\prod_k H^j (M_k; \bbQ) \ar[r]^-\isom  &\prod_k \Hom_\bbQ (H_j (M_k; \bbQ), \bbQ).
		}
$$		
Since $q_k\neq 0$ for all $k$, the map $\prod_k (f_k)_*^\vee$ is injective, and from the diagram, $\prod_k (f_k)^*$ is injective also.  Finally, 
$M \xmaps{\coprod_k f_k} Y$
induces $\prod_k (f_k)^*$ under the isomorphism $H^j(M;\bbQ) \isom \prod_k H^j(M_k; \bbQ)$.
\end{proof}

We now come to the main result of this section.  This result extends a result of Grothendieck \cite[Corollaire 7.2]{Grothendieck-Chern-classes}, which corresponds to the case $X = \{*\}$.  We will say that a space $Y$ has \e{finite type} if $H_j (Y; \bbZ)$ is finitely generated for each $j\geqs 0$.

\begin{theorem} $\label{vanishing2}$
Let $Z$ be a path connected topological space and let 
$$\rho\co X\to \Hom(\pi_1 (Z, z_0), \GL_n (\bbC))$$
be a family of representations parametrized by a space $X$ with $H^j (X; \bbQ) = 0$ for $j>m$.  Assume that either $X$ or $Z$ has finite type.  Then for $i>0$, the  classes 
$$c_{m+i} (E_{\rho} (Z)) \in H^{2m+2i} (X\cross Z; \bbZ)$$
map to zero in $H^{2m+2i} (X\cross Z; \bbQ)$.

In particular, if $\Gamma$ is a finitely generated discrete group and $H^j (X; \bbQ) = 0$ for $j>m$, then for any map $\rho\co X \to \Hom(\Gamma, \GL_n (\bbC))$, the classes $c_{m+i} (E_\rho)\in  H^{2m+2i} (X \cross B\Gamma; \bbZ)$ (with $i>0$) map to zero in $H^{2m+2i} (X\cross B\Gamma; \bbQ)$.
\end{theorem}   

\begin{proof}  In the proof,  $H^*(-)$ will denote rational cohomology. By abuse of notation, we denote the images of Chern classes in de Rham cohomology by $c_i (-)$.

We begin by reducing to the case in which $X$ and $Z$ are closed, connected, smooth manifolds.
  By Lemma~\ref{Thom}, there exist families of smooth, closed manifolds $\{M_k\}_k$ and $\{N_j\}_j$ (with $\dim (M_k)\leqs m$ for each $k$), and maps $M_k \xmaps{f_k} X$ and $N_j \xmaps{g_j} Z$ 
such that $f = \coprod_k f_k$ and $g = \coprod_j g_j$ induce injections on $H^* (-)$.  
By the K\"unneth Theorem for cohomology (Spanier~\cite[Theorem 5.6.1]{Spanier}), the map
$$H^*(X\cross Z) \xmaps{\coprod_{k,j} (f_k \cross g_j)^*} H^* \left(\coprod_{k, j} M_k \cross N_j\right)  \isom \prod_{k, j} H^*(M_k \cross N_j)$$
is injective also.
Hence to show that a class in $x\in H^*(X\cross Z)$ is zero, it suffices to show that $(f_k \cross g_j)^* (x) = 0$ for all $k, j$.
Since $Z$ is path connected and the $N_j$ are smooth manifolds, we may assume that $z_0$ is in the image of $g_j$ for each $j$ (by replacing $g_j$ with a homotopic map if necessary).

By part (1) of Lemma~\ref{compatible}, we have
$$ (f_k \cross g_j)^* ( c_{m+i} (E_{\rho} (Z))) = c_{m+i} (E_{g_j^\# \circ \rho \circ f_k} (N_j)).$$
Assuming the theorem  for the connected, closed, smooth manifolds $N_j$ and $M_k$, we see that $(f_k \cross g_j)^* ( c_{m+i} (E_{\rho} (Z))) = 0$ for each $i>0$ and each $j, k$.
Hence  
$c_{m+i} (E_\rho (Z)) = 0$ in $H^*(X\cross Z; \bbQ)$.

Now we consider the case in which $Z$ and $X$ are  closed, connected, smooth  manifolds.  Let $\Gamma = \pi_1 (Z, z_0)$, and note that $\Gamma$ is finitely generated, so as above we may view $\Hom(\Gamma, \GL_n (\bbC))$ as a real algebraic set.
 By Theorem~\ref{smoothing} and part (2) of Lemma~\ref{compatible}, we may assume without loss of generality that $\rho$ is smooth.

We apply Chern-Weil theory to the fiberwise flat connection $H$ on $E_{\rho} (Z)$ guaranteed by Lemma~\ref{fwf-conn}.  Choose local coordinates $\{x_1,\ldots ,x_m, z_1,\ldots ,z_p\}$ on $X \times Z$ such that the slices $\{x\} \times Z$ are defined locally by setting the $x_i$ coordinates equal to constants. The curvature of  $H$ is defined locally by an $n\times n$ matrix $F$ of differential 2--forms $F_{i,j} \in \Gamma( \wedge^2(T^*(X\cross Z))\otimes \C) $. The fact $H$ is flat along slices $\{x\} \times Z$ implies that $F_{i,j}$ is a (local) section of the subbundle of $\wedge^2(T^*(X\cross Z)\otimes \C)$ spanned by sections of the form $ dx_i \wedge dx_j$ and $dx_i \wedge dz_j$. Consequently, any product of $m+1$ matrix entries from $F$  vanishes. The Chern-Weil form representing $c_{m+i} (E_\rho (Z))$ is  a symmetric polynomial of degree $m+i$ in the entries of $F$.  It follows that for $i>0$ the Chern class $c_{m+i} ( E_{\rho} (Z) )$ is represented in de Rham cohomology by the zero form, and hence $c_{m+i}(E_{\rho} (Z)) = 0$ in $H^*(X\cross Z; \bbR)$ (and also in $H^*(X\cross Z; \bbQ)$). 
\end{proof}

\section{Proof of Theorem \ref{cokernel-AS}}$\label{TAS}$
Let $\Gamma$ be a discrete group.  Consider the natural map
\begin{equation}\label{B}\Hom(\Gamma, \GL_n (\bbC)) \srm{B} \Map_* (B\Gamma, B\GL_n (\bbC))\end{equation}
sending $\rho\co \Gamma \to \GL_n (\bbC)$ to the induced map $B\rho \co B\Gamma\to B\GL_n (\bbC)$ between simplicial classifying spaces.  Note that this map can be defined analogously for any Lie group $G$, and is always continuous (see Ramras~\cite[Section 3]{Ramras-stable-moduli}).

\begin{lemma}\label{B-Erho}
Let $X$ be a space, let $\Gamma$ be a discrete group, and let $G$ be a Lie group.  For every family of representations $\rho\co X \to \Hom(\Gamma, G)$, the principal $G$--bundle $E_\rho\to X \cross B\Gamma$ is classified by the composite
\begin{eqnarray}\label{u}X\cross B\Gamma \xmaps{\rho\cross \Id} \Hom(\Gamma, G) \cross B\Gamma \hspace{1.7in}\\
 \hspace{1.4in} \xmaps{B\cross \Id} \Map_* (B\Gamma, BG) \cross B\Gamma \srm{\mathrm{ev}} BG.\notag
 \end{eqnarray}
\end{lemma}
\begin{proof}
For each point $x\in X$, the homomorphism $\rho_x = \rho(x)$ induces a map
$E(\rho_x) \co E\Gamma\to EG$
between the simplicial models for the universal bundles (we use the conventions described in Ramras~\cite[Section 2]{Ramras-excision}).  This map is equivariant with respect to the right actions of $\Gamma$ and $G$ on these bundles, in the sense that for each $\gamma\in \Gamma$ we have
$E(\rho_x) (e\cdot \gamma) = \left(E(\rho_x) (e)\right) \cdot \rho_x (\gamma)$.

We now have a commutative diagram
\begin{equation}\label{u-diag}
\xymatrix{E_\rho \ar[r]^-{\wt{u}_\rho} \ar[d]   & EG  \ar[d] \\
		X\cross B\Gamma \ar[r]^-{u_\rho}& BG,
	}
\end{equation} 
where $u_\rho$ is the composition (\ref{u}) and the map   $\wt{u}_\rho$ is defined by 
$\wt{u}_\rho([x, e, g]) = (E(\rho_x)(e))\cdot g$.
Equivariance of $E(\rho_x)$ implies that $\wt{u}_\rho$ is well-defined.
Furthermore, $\wt{u}_\rho$ is $G$--equivariant, so Diagram (\ref{u-diag}) is a pullback diagram.
\end{proof}

We write $c_i$ for the $i$th Chern class, considered as a  \e{rational} cohomology class, and we will write $\Ch$ for the Chern character.

\begin{lemma}$\label{Chern}$
Let $X$ be a finite CW complex.  Given $m>0$ and  $x\in H^{2i} (S^m \sm X; \bbQ)$, there exists a  class $\phi \in K^0 (S^m \sm X)$ such that $c_i (\phi) = q x$ for some non-zero rational number $q\in \bbQ^*$, and $c_j (\phi) = 0$ for $j\neq i$.
\end{lemma}
\begin{proof} Since  $\Ch\co K^0 (S^m \sm X) \to  H^{\text{even}} (S^m \sm X; \bbQ)$ is a rational isomorphism, there exists $\phi\in K^0 (S^m \sm X)$ such that $\Ch (\phi) = q x$ for some $q\in \bbQ^*$.  By definition, $\Ch(\phi)$ is a polynomial in the Chern classes of $\phi$, but since all cup products in $H^* (S^m \sm X; \bbQ)$ are zero, $\Ch(\phi)$ is just a rational linear combination of the Chern classes of $\phi$ and the result follows.
\end{proof}

\nd {\bf Proof of Theorem~\ref{cokernel-AS}.}
The map $B\GL_n (\bbC) \to B\GL_{n+1} (\bbC)$ is $(2n+1)$--connected, and analyzing the cofiber sequences $\bigvee S^{k-1} \to X^{(k-1)} \to X^{(k)}$ associated to the skeletal filtration of a $d$--dimensional CW complex $X$, one can show by induction on $d$ that
 for any compatible basepoints and any $0\leqs  m\leqs 2n-d$,  
\begin{equation}\label{maps}\pi_m \Map_* (X, B\GL_n (\bbC)) \isom  \pi_m \Map_* (X, B\GL_\infty (\bbC))  \isom \wt{K}^{-m} (X).\end{equation}

We begin by considering  $B_*$ with respect to the trivial basepoints, namely the trivial representation $1$ and $c=B(1)$, which is the constant map to the basepoint of $B\GL_n (\bbC)$.
For each $i\geqs 1$, choose a $\bbQ$--basis
$$\{a_{ij}\}_j \subset H^{m+2i} (B\Gamma; \bbQ) \isom H^{2m + 2i} (S^m \sm B\Gamma; \bbQ).$$  
By Lemma~\ref{Chern}, there exist classes $A_{ij} \in K^0 (S^m\sm B\Gamma)$ such that 
$$c_{l} (A_{ij}) = \begin{cases} q_{ij} a_{ij} \textrm{ for some } q_{ij}\in \bbQ^*, & \textrm{  if }  l = m+i ,\\
						 0 , &  \textrm{  if }   l\neq m + i.  
						\end{cases}$$
Note that without loss of generality, we may assume $A_{ij} \in \wt{K}^0 (S^m\sm B\Gamma)$.
Set
$$A = \Span_\bbZ (\{A_{ij}\}_{i,j}) \subset \wt{K}^0 (S^m\sm B\Gamma)\isom \pi_m (\Map_*(B\Gamma, B\GL_n (\bbC)), c).$$  
Since $\{\Ch(A_{ij})\}_{i,j} = \{q_{ij} a_{ij}\}_{i,j}$ is  linearly independent over $\bbQ$, $A$ is a free abelian group of rank $\sum_{i= 1}^\infty \beta_{m+2i} (B\Gamma)$, so to complete the proof in this case it suffices to check that $A\cap \Img (B_*) = \{0\}$.

If $x = \sum_{i,j} n_{ij} A_{ij}$ is a non-trivial linear combination of the $A_{ij}$, then 
$$\Ch (x) =  \sum_{i,j} n_{ij} Ch(A_{ij}) = \sum_{i,j} n_{ij} q_{ij} a_{ij}  \neq 0.$$ 
Note that $\Ch(x)$ lies in $\bigoplus_{i= 1}^\infty H^{2m+2i} (S^m\sm B\Gamma; \bbQ)$, and as observed above
$\Ch(x)$ is a rational linear combination of the Chern classes $c_j (x)$.  Hence  for some $i>0$
$$c_{m+i} (x)\in H^{2m+2i} (S^m \sm B\Gamma; \bbQ)$$
 must be non-zero.  
On the other hand, by Lemma~\ref{B-Erho}, classes in the image of $B_*$ classify bundles of the form
 $E_\rho$, and  by Theorem~\ref{vanishing2} we have $c_{m+i} (E_\rho)\in H^*(S^m\cross B\Gamma; \bbQ)$ is trivial for $i>0$.  So  $A\cap \Img (B_*)=\{0\}$.

Now consider general basepoints $\rho$ and $B\rho$.  Using the isomorphisms~(\ref{maps}), we may translate $A$ to a subgroup $A_\rho \subset \pi_m (\Map_*(B\Gamma, B\GL_n (\bbC)), B\rho)$ using Whitney sum with $q^*(E_\rho)$, where $q\co S^m \cross B\Gamma\to B\Gamma$ is the projection. 
Since $E_\rho$ is flat, $c_i (E_\rho) = 0\in H^{2i} (B\Gamma; \bbQ)$, and the same is true for $c_i (q^*(E_\rho))$.  Hence for each non-zero class $a\in A_\rho$ we again have  $c_{m+i}(a) \neq 0$ for some $i>0$, and the rest of the proof proceeds as above.
$\hfill \Box$

\section{The space of flat connections}$\label{flat-sec}$
In this section, we use Theorem~\ref{cokernel-AS} to  study the homotopy types of spaces of flat connections on principal $U(n)$--bundles over  closed, aspherical, smooth manifolds.   

\subsection{Homotopy pullbacks}

We  review some definitions and basic facts.

\begin{definition} The homotopy pullback of a diagram
$X\srm{f} Y \stackrel{g}{\longleftarrow} Z$
is the space
$$\mcP = \mcP(X\to Y\leftarrow Z) = \{(x, \gamma, z) \in X\cross Y^I \cross Z \,|\,  \gamma(0) = f(x), \gamma(1) = g(z)\},$$
topologized as a subspace of $X\cross Y^I\cross Z$, where $Y^I = \Map ([0,1], Y)$.
Given a commutative square 
\begin{equation}\label{square} \xymatrix{ W \ar[r] \ar[d]^-h &Z\ar[d]^-g\\ X \ar[r]^-f & Y,}\end{equation}
there is a natural map $W\to \mcP$.  If this map is a weak equivalence (meaning that it induces an isomorphism on homotopy groups for all compatible choices of basepoints), we call the diagram (\ref{square})
\emph{homotopy cartesian}.
\end{definition}

\begin{definition} The homotopy fiber of a continuous map $f\co X \to Y$ at a point $y\in Y$ is the space
$\hofib_y (f) = \mcP(X\srm{f} Y \leftarrow \{y\})$.
\end{definition}

For any compatible choice of basepoints, there is a long exact sequence 
\begin{equation}\label{LES}\cdots \maps \pi_* \hofib_y (f) \maps \pi_* X \maps \pi_* Y \srm{\partial} \pi_{*-1} \hofib (f) \maps \cdots
\end{equation}
resulting from the fact that $\hofib_y (f)$ is the fiber (over $y$) of the fibration
$$P_f = \mcP (X\srm{f} Y \stackrel{\Id}{\lmaps} Y) \srm{p_f} Y$$
 associated to $f$ (here $p_f$ is the projection to the last factor).
If $Y$ is disconnected (which is often the case for representation spaces), then $\hofib_y (f)$ depends only on the inverse image, under $f$, of the path component $[y]$ of $Y$ containing $y$: 
\begin{equation}\label{comp}\hofib_y(f) = \hofib_y (f|_{f^{-1} ([y])}).\end{equation}

\begin{proposition}$\label{h-cart}$
The square (\ref{square}) is homotopy cartesian if and only if 
the natural map $\hofib_x (h) \to \hofib_{f(x)} (g)$ is a weak equivalence for each $x\in X$.
\end{proposition}

This may be proven by noting that there is a natural homeomorphism
$$\hofib_{(x, \gamma, z)} (W\to \mcP) \homeo \hofib_{(z, \overline{\gamma})} \left(\hofib_x (h) \to \hofib_{f(x)} (g)\right)$$
for each point $(x,\gamma, z)\in \mcP$.  Here $\overline{\gamma}$ denotes the path $\overline{\gamma} (t) = \gamma(1-t)$.

The following result follows from Hirschhorn~\cite[Theorems 13.1.11 and 13.3.4]{Hirschhorn}.

\begin{proposition}$\label{htpy-inv-fibers}$ A diagram of spaces
$$\xymatrix{ X\ar[r] \ar[d]^f & Y \ar[d]^g & Z \ar[l] \ar[d]^h \\ X'\ar[r] & Y'   & Z'\ar[l] }$$
induces a map $\Phi: \mcP \to \mcP'$ between the homotopy pullbacks of the rows, and if $f, g$, and $h$ are weak equivalences, so is $\Phi$.
\end{proposition}

\subsection{Maps between classifying spaces and flat connections}$\label{flat}$

Consider a principal $U(n)$--bundle $E\srt{\pi} M$ over a closed, connected, aspherical, smooth manifold $M$ of dimension $d$.  Fix basepoints $e_0\in E$, $m_0\in M$ with $\pi(e_0) = m_0$.  We will (implicitly) use connections on $E$ of Sobolev class $L^p_k$ for some fixed constants $p\in \bbR$, $k\in \bbN$ with $\min (d/2, 4/3)  < p < \infty$, $k\geqs 2$, and $kp > m$.  Similarly, we use gauge transformations of $E$ of Sobolev class $L^p_{k+1}$, so that the based gauge group $\G_0 (E)$ acts continuously on the space $\fc(E)$ of flat unitary connections, with quotient homeomorphic to $\Hom(\pi_1 M, U(n))$ (this uses the Strong Uhlenbeck Compactness Theorem).  This quotient map may be identified with the map
$$\mH \co \fc (E) \maps \Hom(\pi_1 M, U(n)),$$
sending a flat connection $A$ to the holonomy representation of $A$ computed at $e_0\in E$.
A more detailed discussion  may be found in Ramras~\cite[Section 3]{Ramras-surface}.

Let $\Gamma = \pi_1 (M, m_0)$. 
We want to compute the homotopy fibers of the map $B$ studied in Section~\ref{TAS}.  Choose inverse homotopy equivalences $f\co M\to B\Gamma$ and $g\co B\Gamma\to M$ satisfying $f(m_0) = *$ and $g(*) = m_0$ (where $*\in B\Gamma$ is the canonical basepoint) and choose $\phi_E\co B\Gamma\to BU(n)$ such that
$g^*(E)\isom \phi_E^* (EU(n))$.

\begin{theorem}$\label{hofib}$
In the situation described above, there is a weak equivalence
$$ \hofib_{\phi_E} (B) \heq\flatc (E).$$
In particular, $\flatc (E)$ is independent of the choice of Sobolev norm up to weak equivalence.
\end{theorem}

\begin{proof}  We will construct a commutative diagram of the form
\begin{equation}\label{T-diag}
\xymatrix{\flatc (E) \ar[rr]^(.45)\T \ar[d]^\mH & & \Map^{U(n)}_* (E, EU(n)) \ar[d]^{q} \\
	 	\Hom^E (\Gamma, U(n)) \ar[r]^(.45){B} &  \Map_*^{g^*E} (B\Gamma, BU(n)) \ar[r]^(.52){f^*}_{\heq} 
					& \Map_*^E (M, BU(n)).
	     }
\end{equation}
Here $\Map_*^E$ and $\Map_*^{g^*E}$ denote the path components of  the based mapping spaces  consisting of maps classifying the bundles $E$ and $g^* E$ (respectively), and $\Map^{U(n)}_*$ is the space of based, $U(n)$--equivariant maps.
  The subspace $\Hom^E \subset \Hom$ consists of those representations inducing  bundles isomorphic to $g^* (E)$ (in other words, $\Hom^E$ is the inverse image of $\Map_*^E$ under $f^*\circ B$).   The map $q$ sends an equivariant map $E\to EU(n)$ to the induced map $M\to BU(n)$.
	
We now define the map $\T$, in analogy with Ramras~\cite[Proof of Theorem 3.4]{Ramras-stable-moduli}.  Given a flat connection $A\in \fc (E)$, let $\rho_A = \mH(A)$ be the holonomy representation of $A$ (computed at the basepoint $e_0\in E$).  By Lemma~\ref{fwf-conn}, if we choose a basepoint $\wt{m}_0\in \wt{M}$ lying over $m_0\in M$, the bundle 
$E_{\rho_A}  (M)  \to M$
acquires a basepoint $[\wt{m}_0, I_n]$ and admits a canonical flat connection $A_\rho$ whose holonomy, computed at $[\wt{m}_0, I_n]$, is precisely $\rho$. 
In addition, there is a canonical isomorphism of principal $U(n)$--bundles
$E\srm{\phi_A} E_{\rho_A} (M)$
such that $\phi_A (e_0) = [\wt{m}_0, I_n]$ and $\phi_A^* (A_\rho) = A$.

The homomorphism $\rho_A\co \Gamma \to U(n)$ induces a $U(n)$--equivariant map 
$$E(\rho_A)\co E\Gamma\maps EU(n)$$
between the simplicial models for these universal bundles, and $E(\rho_A)$ covers $B(\rho_A)$.
Next, $f \co M\xmaps{\heq} B\Gamma$ is covered by a based map $\wt{f} \co \wt{M} \to E\Gamma$ of principal $\Gamma$--bundles.
Using (\ref{u-diag}) and part (1) of Lemma~\ref{compatible}, we may now define maps of principal $U(n)$--bundles
$E_{\rho_A} (M) \to E_{\rho_A} (B\Gamma)\to EU(n)$ whose composition $\psi_A$ covers
 $B(\rho_A) \circ f$.  We define
$$\T (A) = \psi_A \circ \phi_A \in \Map^{U(n)} (E, EU(n)).$$
 For trivial bundles $E\isom M\cross U(n)$, continuity of  
$\T$ 
was proven in Ramras~\cite[Theorem 3.4]{Ramras-stable-moduli} and the general case is similar.

Now consider the following commutative diagram:
\begin{equation}\label{pb-comp}\xymatrix{\Hom^E (\Gamma, U(n)) \ar[r]^-{B} \ar[d]^\Id &  \Map_*^{g^*E} (B\Gamma, BU(n)) \ar[d]^{f^*} & \{\phi_E\} \ar[l] \ar[d]\\
		\Hom^E (\Gamma, U(n)) \ar[r]^-{f^* \circ B} \ar[d]^\Id &  \Map_*^E (M, BU(n)) \ar[d]^\Id &  \{\phi_E\circ f\} \ar[l] \ar[d]^i\\
		\Hom^E (\Gamma, U(n)) \ar[r]^-{f^* \circ B}  &  \Map_*^E (M, BU(n)) & \Map^{U(n)}_* (E, EU(n)) \ar[l]_{q}.
		}
\end{equation}
The map $i$ sends $\phi_E\circ f$ to  a map $\wt{\phi_E\circ f}\in \Map^{U(n)}_* (E, EU(n))$ covering $\phi_E\circ f\co M\to BU(n)$; such a map exists because (by the Bundle Homotopy Theorem)
$$(\phi_E \circ f)^* (EU(n)) = f^* (\phi_E^* (EU(n))) \isom f^* (g^* E) \isom E.$$
By Gottlieb~\cite[Theorem 5.6]{Gottlieb}, the space $\Map^{U(n)}_* (E, EU(n))$ is weakly  contractible, so all vertical maps 
in Diagram (\ref{pb-comp}) are weak equivalences and by Proposition~\ref{htpy-inv-fibers}, the induced maps between the homotopy pullbacks of the rows are weak equivalences as well.  
By (\ref{comp}), the homotopy pullback of the top row is just $\hofib_{\phi_E} (B)$, and the bottom row is the square (\ref{T-diag}) (minus its upper left corner).  To complete the proof, then, it will suffice to show that (\ref{T-diag}) is homotopy cartesian.

The vertical maps in  (\ref{T-diag}) are fibrations: for $\mH$, see Mitter--Viallet~\cite{Mitter-Viallet} or Fine--Kirk--Klausen~\cite{FKK},
and for $q$, see Gottlieb~\cite[Proposition 5.3]{Gottlieb}.  The induced map between the  fibers of the vertical maps is the inclusion
of the based Sobolev  gauge group into the continuous gauge group.  These spaces are spaces of sections of the adjoint bundle $E\cross_{U(n)} U(n)^{\mathrm{Ad}}$, so this inclusion map is a weak equivalence by general approximation results for spaces of sections.  By Proposition~\ref{h-cart},  Diagram (\ref{T-diag}) is homotopy cartesian.  (For fibrations, the inclusion of the fiber into the homotopy fiber is a homotopy equivalence~\cite[Proposition 4.65]{Hatcher}.)
\end{proof}

\begin{proof}[Proof of Corollary~\ref{fh}]  This follows from Theorem~\ref{cokernel-AS} together with the long-exact sequence in homotopy associated to the (homotopy) fiber sequence
$$\fc(E) \maps \Hom^E (\Gamma, U(n))\srm{B}  \Map_*^{g^*E} (B\Gamma, BU(n))$$
established in Theorem~\ref{hofib}.  Note that for the statement regarding $\pi_0 (\fc(E))$, we use the fact that in a fibration sequence $(F, f_0)\to (E, f_0)\srt{p} (B, b_0)$, if $\partial (\gamma) = \partial (\eta)$ for some $\gamma, \eta\in \pi_1 (B, b_0)$, then 
$\gamma^{-1} \eta \in \Img \left(\pi_1 (E, f_0) \to \pi_1 (B, b_0)\right)$.
\end{proof}

\begin{remark} Choosing a basepoint $A_0\in \fc(E)$, Theorem~\ref{hofib} yields a long exact sequence in homotopy, whose boundary map
is the composition  
\begin{eqnarray*}
 \pi_* \bMap (B\Gamma, BU(n)) \mapsmo \limits_{\isom}^{f^*}  \pi_* \bMap (M, BU(n)) \hspace{1.8in}\\
\hspace{1.3in} \mapsmo  \limits_{\isom}^{\partial_q} \pi_{*-1} \Map^{U(n)}_* (E, U(n)) \mapsmo \limits_{\isom}^{i_*^{-1}}  \pi_{*-1} \G_0 (E) \srm{j_*} \pi_{*-1} \fc (E),
\end{eqnarray*}
where $\partial_q$ is the boundary map for $q$  (see (\ref{T-diag})), 
$i$ is the natural inclusion, and $j$ is the inclusion of the orbit of $A_0$.
Hence the classes in $\pi_* \fc (E)$ described in Theorem~\ref{fh} are all in the image of
$\pi_* \G_0 (E) \srm{j_*} \pi_* \fc(E)$.

Our methods can also be used to identify the boundary map for the holonomy fibration $\mcH$ with the map $B_*$, via the isomorphisms 
\begin{eqnarray*} \pi_{*-1} \G_0 (E) \mapsmo \limits^{i_*}_{\isom}   \pi_{*-1} \Map^{U(n)}_* (E, U(n)) \mapsmo \limits_{\isom}^{\partial_q^{-1}}  \pi_{*} \Map^E_* (M, BU(n))\\
\mapsmo \limits^{\isom} \pi_{*} \Map^{g^*E}_* (B\Gamma, BU(n)).
\end{eqnarray*}

\end{remark}

\end{document}